\documentclass[12pt]{article}
\usepackage{amsmath,amsthm,amsfonts,amssymb,amscd}
\usepackage{cite}
\usepackage{soul}
\usepackage{xcolor}
\usepackage{bbm}
\renewcommand{\emptyset}{\varnothing}

\newtheorem{thm}{Theorem}[section]
\newtheorem{exampleee}{Example}[section]
\newtheorem{lem}{Lemma}[section]
\newtheorem{dfn}{Definition}[section]
\newtheorem{propn}{Proposition}[section]
\usepackage{amsmath}
\usepackage{cite}

\newcommand{\nexto}{\kern -0.54em}

\newcommand{\dZ}{{\cal Z \kern -0.7em Z}}
\newcommand{\dC}{{\rm\hbox{C \kern-0.8em\raise0.2ex\hbox{\vrule height5.4pt
width0.7pt}}}}
\newcommand{\dQ}{{\rm\hbox{Q \kern-0.85em\raise0.25ex\hbox{\vrule height5.4pt
width0.7pt}}}}

\usepackage{epsfig}
\usepackage{latexsym}

\newcommand{\CC}{\mathcal{C}}
\newcommand{\NN}{\mathbb{N}}

\newcommand{\RR}{\mathbbm{R}}
\begin{document}

\title{The inexact projected gradient method for quasiconvex vector optimization problems}
\author{
J.Y. Bello Cruz\footnote{IME - Federal University of Goi\'as,
Goi\^ania, GO, Brazil. E-mail: yunier.bello@gmail.com} \and G.C.
Bento\footnote{IME - Federal University of Goi\'as, Goi\^ania, GO,
Brazil. E-mail: glaydston@ufg.br} \and G. Bouza Allende
\footnote{Faculty of Mathematics and Computer Science, University of
Habana, Cuba. E-mail: gema@matcom.uh.cu}\and R.F.B.
Costa\footnote{IME - Federal University of Goi\'as, Goi\^ania, GO,
Brazil. E-mail: raynerbarbosa@gmail.com} }

\maketitle

\begin{abstract}
\noindent Vector optimization problems are a generalization  of
multiobjective optimization in which the preference order is related
to an arbitrary closed and convex cone, rather than the nonnegative
octant. Due to its real life applications, it is important to have
practical solution approaches for computing. In this work, we
consider the inexact projected gradient-like method for solving
smooth constrained vector optimization problems. Basically, we prove
global convergence of any sequence produced by the method to a
stationary point assuming that the objective function of the problem
is $K$-quasiconvex, instead of the stronger $K$-convexity assumed in
the literature.
\end{abstract}

\noindent{\bf Keywords:} Gradient-like method; Vector optimization;
$K$-- quasiconvexity.

\medskip

\noindent{\bf Mathematical Subject Classification (2008):} 90C26;
90C29; 90C31.

\section{Introduction}\label{sec_intro}
In many applications, it is desired to compute a point such that
there exists not a preferred option. Sometimes this preference is
mathematically described by means of a cone $K$ and a function $F$,
i.e. a point $x$ is preferred to $y$ if the difference of the
evaluations of the function belongs to the cone. This defines the
vectorial optimization problem $\min F(x)$ in which the order is
given the cone $K$. In this paper we assume that $F$ is a
quasiconvex function with respect to the cone $K$, as in many
micro-economical models devoted to maximize the utilities, which
usually are quasi-concave functions.

A popular technique for solving vectorial problems  is to scalarize
the function $F$; see \cite{gme,jahn0} and the references therein.
Many disadvantages on this scheme has been reported in, for
instance, \cite{luc}. Furthermore in the quasiconvex case, the
scalarizations may lead to solve non-quasiconvex models (sum of
quasiconvex functions may not be quasiconvex). An appealing solution
approach is to use descent directions like Gradient method  for
multiobjective and vector optimization problems; see, for instance,
\cite{fliege-svaiter,fu-mauricio,grana-iusem,grana-svaiter}. For
quasiconvex models, the convergence analysis has been obtained in
\cite{yunier,Bento2012-2,Bento2013-1,Bento2012-1}.

In the present paper we extend the convergence result of
\cite{yunier} for vector optimization and study the inexact version
of the projected gradient method. Our method was inspired by the
method proposed in \cite{Fukudainexact} and uses similar idea
exposed in \cite{yunier}. Assuming existence of solutions and
quasiconvexity of the vector function, we prove that every sequence
generated converges to a stationary point of the vector optimization
problem.

This article is organized as follows. Section \ref{secdef} contains
some basic definitions and preliminary material. In Section
\ref{secalgteste}, we present the inexact projected gradient method
for vectorial optimization. Section \ref{CR} provides the
convergence analysis of the method. Finally, in Section
\ref{remarksss} we give the final remarks.

\section{Basic definitions and preliminary material}\label{secdef}
In this section, we present the vector optimization problem  as well
as some definitions, notations and basic properties, which are used
throughout of this paper. For more details; see,
\cite{quasiconvex,Fukudainexact,grana-svaiter}.

Let $K \subset \RR^m$ be a nonempty closed, convex and pointed cone.
The partial order $``\preceq_{K}"$ induced by $K$ in $\RR^m$ is
defined as $u\preceq_{K}v$ if $v-u\in K$. A set $Y \subset \RR^m$ is
$K$--bounded if there exists $ z \in \RR^m$ such that, for all $y\in
Y$, $z \preceq_{K}y$. Assuming that $\mbox{int} (K)$ is nonempty,
$u\prec_{K}v$ if $v-u \in\mbox{int}(K)$. As reported in \cite[Remark
2.2]{quasiconvex}, if \mbox{int}$(K)$ is nonempty, the partial order
is directed, i.e, for all $y_1,y_2 \in \RR^m,$ there exists
$z\in\RR^m$ such that $y_1\preceq_{K}z$ and $y_2\preceq_{K}z$.

Given $K$,  its positive polar cone is $K^{\star}:= \{y\in \RR^m ;
\langle y,x \rangle \geq 0 \ \mbox{for all}\ x \in K\}$  and its
generator is  a compact set $G$ such that $K$ is the cone generated
by its convex hull.  As pointed out in \cite[Remark
3.2]{quasiconvex}, if $\mbox{int}(K)$ is nonempty, then $ K^{\star}
= \mbox{{co}}(\mbox{{conv}}(\mbox{{extd}}(K^{\star}))), $ remind
that $d \in \mbox{extd}(K)$ is an extreme direction if $d \neq 0$
and $d = d_1 + d_2$ for some $d_1, d_2 \in K$   implies that $d_1,
d_2 \in \RR_{+}d$.

Given $F:\RR^n\to {\mathbb R}^m$, a $\CC^1$ function and $C\subseteq
\RR^n$ a nonempty, closed and convex set, we consider the problem of
finding a  weakly efficient point of $F$ in $C$, i.e., a point
$x^*\in C$ such that there exists no other $x\in C$ with
$F(x)\prec_{K} F(x^*)$.  We denote this constrained problem as
\begin{equation}\label{Prel_P}
 \min_{K} F(x), \quad
\mbox{s.t. } x\in C.
\end{equation}
We  denote $J_F(x)$ as the Jacobian matrix of $F$ at $x$ and
$C-x^*=\{y-x^*: \; y\in C\}$. We say that $v$  is a descent
direction at $x \in C$ if there is not $v\in C - x$ such that
$J_{F}(x)v\prec_{K}0$. This leads to the definition of stationarity,
i.e, $x^*\in C$ is a stationary point if
\begin{equation}\label{stationarity}
-\mbox{int}(K)\cap J_F(x^*)(C-x^*)=\emptyset.
\end{equation}   For characterizing stationary points,   we consider
$\varphi:\RR^m\rightarrow\RR$, $$\varphi(y):=\mbox{max}\langle
y,\omega\rangle,\quad\mbox{ s.t. }\omega \in G.$$ As reported in
\cite{grana-svaiter}, $ -K=\{y\in \RR^m: \varphi(y)\leq0\} \ \
\mbox{and} \ \ -\mbox{int}(K)=\{y\in \RR^m: \varphi(y)<0\}$.
Furthermore, in \cite[Proposition 2]{grana-iusem}, it is shown that
$\varphi$  is positively homogeneous of degree $1$, subadditive and
a Lipschitz continuous with Lipschitz constant $L=1$. If $y\prec_K
z$ $(y\preceq_K z$,  respectively), $\varphi(y)<\varphi(z)$
($\varphi(y)\leq\varphi(z)$). Now we define
$h_{x}:C-x\rightarrow\RR$ by
$$h_{x}(v):=
\hat{\beta}\varphi(J_{F}(x)v) +\|v\|^2/2,$$ where $\hat{\beta} > 0$
and the following constrained parametric optimization problem:
\begin{equation} \label{direction exact} \mathop{\min} h_{x}(v), \quad \mbox{ s.t. }v\in C-x.\end{equation}
This problem has only one solution, namely $\bar{v}$, and it
fulfills that $ \bar{v}= P_{C-x}(-\hat{\beta}{J_{F}(x)}^{T}\omega)$
for some $\omega \in \mbox{conv}(G)$; see \cite{grana-iusem}. So, we
have the following:

\begin{dfn}\label{defsddf1} The projected gradient
direction function of $F$ is defined as $v:C\rightarrow\RR^n$, where
$v(x)$ is the unique solution of Problem (\ref{direction exact}).
The optimal value  function associated to (\ref{direction exact}) is
$\theta:C\rightarrow \RR$, where $\theta(x):= h_x(v(x))$.
\end{dfn}

\begin{lem}\label{lemma2012-1}
 $x$ is a stationary point of $F$, if and only if $\theta(x)=0$ and $v(x)=0$.  $v(x)$ and $\theta(x)$ are continuous functions.
\end{lem}
\begin{proof} For the first part see \cite[Proposition 3]{grana-iusem}.  For the continuity of $v(x)$;
see \cite[Proposition 3.4]{fu-mauricio}. The second part is a direct
consequence of this fact. \end{proof} Now we consider the inexact
case. Let us present the concept of approximate directions.
\begin{dfn}\label{defapprox}
Let $x \in C$ and $\sigma \in [0,1)$. A vector $v\in C-x$ is a
$\sigma$--approximate projected gradient direction at $x$ if $
h_{x}(v)\leq(1-\sigma)\theta(x). $
\end{dfn}
%\begin{remark}\label{obs}$\;$
%\begin{enumerate}
%\item [(i)] $v(x)$ is a $\sigma$-approximate at $x$ for  $\sigma \in [0,1)$ and the unique $0$-approximate direction at $x$.
%\item [(ii)] If $x \in C$  is  non-stationary and $0\leq\sigma<1$ then $\sigma$-approximate directions are descent directions.  \end{enumerate}
%\end{remark}
A particular class of $\sigma$--approximate directions for $F$ at
$x$ is given by the so called {\it scalarization compatible} (or
simply s--compatible) directions, i.e., those $v\in\RR^n$  such that
 \begin{equation}\label{eq:2012-350}
 v=P_{C-x}(-\hat{\beta}{J_{F}(x)}^{T}\omega), \mbox{ for some }\omega\in \mbox{conv}( G).
 \end{equation}
Relations between s--compatible and $\sigma$--approximate directions
can be found in \cite{Fukudainexact}. The convergence of the method
is obtained using the following definition and results.
\begin{dfn}[Definition 4.1 of \cite{Entropy}]\label{def:tec1}
Let $S$ be a nonempty subset of $\RR^n$. A sequence $(x^k)_{k\in
\NN}$ in $\RR^n$ is said to be quasi-Fej\'er convergent to $S$ if
and only if for all $x \in S$ there exist $k_0\geq 0$ and a sequence
$(\delta_k)_{k\in \NN}$ in $\RR_+$ such that
$\sum_{k=0}^\infty\delta_k<\infty$ and $$\| x^{k+1}-x\|^2 \leq \|
x^{k}-x\|^2 +\delta_k,$$ for all $k\geq k_0$.
\end{dfn}
\begin{lem}[Theorem 4.1 of \cite{Entropy}]\label{fejer}
If $(x^k)_{k\in \NN}$ in $\RR^n$ is {\rm quasi-Fej\'er} convergent
to some set $S$, then:
\begin{itemize}
\item [i)] The sequence $(x^k)_{k\in \NN}$ is bounded;

\item [ii)] If there exists an accumulation point, $x$, of the sequence $(x^k)_{k\in \NN}$ belongs to $S$, then $\{x^{k}\}$ is convergent to $x$.
\end{itemize}
\end{lem}
We end with a brief introduction to quasiconvexity in the vectorial
framework.
\begin{dfn}\label{def5}
The vector function $F:\RR^n \rightarrow \RR^m$ is said to be
$K$--quasiconvex if for all $y\in \RR^m$ the level set
$L_{F}(y)=\{x\in \RR^n: F(x)\preceq_{K}y\}$ is convex.
\end{dfn}
The following characterization will be useful.
\begin{thm}\label{characterization}
Assume that $(\RR^m, \preceq_{K})$ is partially ordered. Then $F$ is
$K$--quasiconvex if and only if
 $\langle d,F(.) \rangle:\RR^n\rightarrow\RR$ is quasiconvex for every extreme direction $d\in K^{\star}.$
\end{thm}
\begin{proof}
As already remarked, \mbox{int}$(K)$ is a nonempty set, $K^{\star}$
is the conic hull of the closed convex hull of extd$(K^{\star})$ and
$(\RR^m, \preceq_{K})$ is directed. Combining these two facts, the
desired result follows from \cite[Theorem $3.1$]{quasiconvex}.
\end{proof}
\section{Inexact projected gradient algorithm}\label{secalgteste}
This part is devoted to present the method  proposed  in
\cite{Fukudainexact} and some properties of it. Fix $\hat{\beta}>0,$
$\delta\in (0,1),$ $\tau>1$ and $\sigma\in [0,1)$. The inexact
projected gradient method is defined as follows.

\vspace{0.035in}

\noindent {\bf Initialization:}  Take $x^0\in C$.

\vspace{0.025in}

\noindent  {\bf Iterative step:} Given $x^k$, compute a
$\sigma$-approximate direction $v^k$ at $x^k.$

\noindent  If $h_{x^k}(v^k)=0$, then stop. Otherwise compute
\begin{equation}\label{Armijo}
j(k):=\mbox{min}\left\{ j\in\mathbb{Z_{+}}: \ F(x^k +\tau^{-j}v^k
)\preceq_{K} F(x^k) + \delta\tau^{-j}J_{F}(x^k)v^k \right\}.
\end{equation}
\noindent  Set \ $t_{k}= \tau^{-j(k)}$ \ and \
$x^{k+1}=x^k+t_{k}v^k$.

\vspace{0.035in}

If $m=1$ and $\sigma=0$, the method becomes the classical exact
projected gradient method. In the inexact unconstrained case, we
retrieve the method introduced  in \cite{grana-svaiter}. The
following  holds.

\begin{propn}\label{boa definicao}
$h_{x^k}(v^k)=0,$  implies the stationarity of $x^k.$ Let $\delta
\in (0,1)$, $x^k\in C$ and let $v^k$ be a descent direction. Then,
there exits $\bar{\gamma}>0$ such that (\ref{Armijo}) holds for all
$\gamma\in [0, \bar{\gamma}]$, i.e.,
$$F(x^k +\gamma v^k )\preceq_{K} F(x^k) +
\delta\gamma J_{F}(x^k)v^k.
$$ So, the Armijo rule
is well defined.
\end{propn}
\begin{proof}
For the first part note that  if $h_{x^k}(v^k)=0,$ then, by the
definition of $\sigma$-approximation, $\theta(x^k)\geq0$, but as
$\theta(x^k)\leq0,$ so $\theta(x^k)=0,$ concluding that $x^k$ is a
stationary point. On the other hand, if $x^k$ is stationary, then
$\theta(x^k)=0,$ and therefore $h_{x^k}(v^k)=0.$ The last part
follows from \cite[Proposition 1]{grana-iusem}. \end{proof}
\section{Convergence analysis}\label{CR}
In this section, we show the global convergence of the inexact
projected gradient method. If the method stops after a finite number
of iterations, it computes a stationary point as desired. So, we
will assume that $(x^k)_{k\in \NN}$, $(v^k)_{k\in \NN}$,
$(t_k)_{k\in \NN}$ are the infinite sequences generated by the
inexact projected gradient method. From \cite[Lemma
$3.6$]{Fukudainexact}, we recall that if $(F(x^k))_{k\in \NN}$ is
$K$--bounded, then
\begin{equation}\label{somatorio}
 \sum_{k=0}^{\infty}t_k|\langle \omega,J_{F}(x^k)v^k\rangle| < +\infty,\mbox{ for all }\omega \in\mbox{conv}(G).
\end{equation}
\begin{propn}\label{feasible}
Sequence $(x^k)_{k\in \NN}$ is feasible and $F(x^k)-F(x^{k+1})\in K$
for all $k$.
\end{propn}
\begin{proof}
The feasibility is a consequence of the definition of the method and
the $K$--decreasing property follows from (\ref{Armijo}).
\end{proof} Under differentiability, the convergence is
obtained in \cite[Theorem $3.5$]{Fukudainexact}  as follows.
\begin{propn}\label{iusmau}
Every accumulation point, if any, of $(x^k)_{k\in \NN}$ is a
stationary point of Problem (\ref{Prel_P}).
\end{propn}

\noindent In what follows we present the main novelty of this paper.
For the convergence of the method we need the following hypotheses.
\\{\bf Assumption 1.} $T\neq\emptyset$, where $
T:=\left\{ x\in C\colon F(x)\preceq_{K} F(x^k), \; k=0,1,
\ldots,\right\}. $
\\{\bf Assumption 2.} Each $v^k$ of the sequence $(v^k)_{k\in \NN}$ is  scalarization compatible,
i.e., exists a sequence $(\omega^k)_{k\in \NN}\subset\mbox{conv}
(G)$ such that
\[
v^k=P_{C-x^k}(-\hat{\beta}{J_{F}(x^k)}^{T}\omega^k),\qquad
k=0,1,\ldots.
\]
The convergence of several methods for solving vector optimization
problems is usually obtained under Assumption 1; see \cite{yunier,
Bento2012-2,
Bento2012-1,fliege-svaiter,fu-mauricio,grana-iusem,grana-svaiter}.
Although the existence of a weakly efficient solution does not imply
that $T$ is nonempty, it is closely related with the completeness of
the Im$(F)$, which ensures the existence of efficient points; see
\cite{luc}. Moreover, if  the sequence $(x^k)_{k\in \NN}$ has an
accumulation point, then $T$ is nonempty; see
\cite{Bento2012-2,Bento2012-1}.

Assumption $2$ holds if  $v^k$ is the exact gradient projected
direction at $x^k$. It was also used in \cite{Fukudainexact} for
proving the full convergence of the sequence generated by the method
in the case that $F$ is $K$--convex. \\From now on, we will assume
that Assumptions $1$-$2$ hold. We start with the following result

\begin{lem}\label{lemma:tec1}
For each $\hat{x}\in T$ and $k\in \mathbb{N}$, it holds that $$
\langle v^k,\hat{x} - x^k \rangle \geq  \hat{\beta}\langle
J_{F}(x^k)^{T}\omega^k,v^k\rangle + \|v^k\|^2. $$
\end{lem}
\begin{proof}
Take $k\in \mathbb{N}$ and $\hat{x}\in T$. As $v^k$ is
$s$-compatible at $x^k$, then
$v^{k}=P_{C-x^k}(-\hat{\beta}J_{F}(x^k)^{T}\omega^k)$, for some
$\omega^k \in \mbox{conv}(G)$. As $v^k$ is a projection, $
 \langle-\hat{\beta}J_{F}(x^k)^{T}\omega^k - v^k, v - v^k\rangle \leq 0,$ for all $v\in C - x^k.
$ In particular, for $v = \hat{x} - x^k,$ we obtain $
\langle-\hat{\beta}J_{F}(x^k)^{T}\omega^k - v^k, \hat{x} - x^k -
v^k\rangle \leq 0 .$ So, from the last inequality, we get
\begin{equation} \label{eq:tec1}
\langle v^k,\hat{x} - x^k \rangle \geq -\hat{\beta}\langle
J_{F}(x^k)^{T}\omega^k,\hat{x} - x^k \rangle  + \hat{\beta}\langle
J_{F}(x^k)^{T}\omega^k,v^k\rangle + \|v^k\|^2.
\end{equation}
Since $F$ is $K$--quasiconvex, by Theorem \ref{characterization},
for each $d\in \mbox{extd}(K^{\star})$,  $\langle d,F \rangle: \RR^n
\rightarrow \RR$ is a quasiconvex function. As
co(conv(extd$(K^{\star})$)) $=K^{\star}$ and $\omega^k \in
\mbox{conv}(G)\subset K^{\star},$  $\omega^k
=\sum_{\ell=1}^p\gamma_{\ell}^{k}d_{\ell},$ where
$\gamma_{\ell}^{k}\in \mathbb R_+ $ and $d_{\ell} \in\mbox{extd}(K
^{\star}),$ for all $1\leq \ell \leq p$. Therefore,
$$
\hat{\beta}\langle J_{F}(x^k)^{T}\omega^k,\hat{x} - x^k\rangle =
\hat{\beta}\langle
J_{F}(x^k)^{T}\sum_{\ell=1}^p\gamma_{\ell}^{k}d_{\ell},\hat{x} -
x^k\rangle = \hat{\beta}\sum_{\ell=1}^p\gamma_{\ell}^{k}\langle
J_{F}(x^k)^{T}d_{\ell},\hat{x} - x^k\rangle. $$ As $ \hat{x}\in T$,
we have $F(x^k) - F(\hat{x}) \in K .$ So, $\langle d_\ell ,F(x^k) -
F(\hat{x})\rangle \geq 0$, for all $d_\ell \in
\mbox{extd}(K^{\star})$. But
 as $\langle d_\ell ,F \rangle$ is a real-valued, quasiconvex differentiable function, it follows that
 $
 \langle J_{F}(x^k)^{T} d_{\ell},\hat{x} - x^k \rangle \leq0.
 $
This implies that $ \hat{\beta}\langle
J_{F}(x^k)^{T}\omega^k,\hat{x} - x^k\rangle \leq 0. $ Now, the
result follows from combining of the last inequality with
(\ref{eq:tec1}). \end{proof} Next lemma  presents the quasi-Fej\'er
convergence.
\begin{lem}\label{Fejconv2}
Suppose that $F$ is $K$--quasiconvex. Then, the sequence
$(x^k)_{k\in \NN}$ is quasi-Fej\'er convergent to the set $T$.
\end{lem}

\begin{proof}
Since $T$ is nonempty, take $\hat{x}\in T$ and fix $k\in
\mathbb{N}$. Using the definition of $x^{k+1}$, after some algebraic
work, we are lead to
\begin{equation}\label{eq:tec3}
\|x^{k+1} - \hat{x}\|^2=\|x^{k} - \hat{x}\|^2 +  \|x^{k+1}
-x^{k}\|^2 - 2t_k\langle v^k,\hat{x} - x^{k}\rangle.
\end{equation}
Using Lemma \ref{lemma:tec1}, recall that $t_{k}\in(0,1)$, we get
\begin{equation}\label{eq:tec2}
\|x^{k+1} - x^k\|^2 - 2t_{k}\langle v^k,\hat{x} - x^k\rangle \leq
t_{k}\|v^k\|^2 -2t_{k}(\hat{\beta}\langle
J_{F}(x^k)^{T}\omega^k,v^k\rangle + \|v^k\|^2).
\end{equation}
On the other hand, $ t_{k}\|v^k\|^2 -2t_{k}(\hat{\beta}\langle
J_{F}(x^k)^{T}\omega^k,v^k\rangle + \|v^k\|^2) \leq
-2t_{k}\hat{\beta}\langle J_{F}(x^k)^{T}\omega^k,v^k\rangle, $
Recalling  that $\alpha\leq |\alpha|$,   from (\ref{eq:tec2}), we
obtain $ \|x^{k+1} - x^k\|^2 - 2t_{k}\langle v^k,\hat{x} -
x^k\rangle \leq  2t_{k} |\hat{\beta}\langle
J_{F}(x^k)^{T}\omega^k,v^k\rangle|. $ Combining last inequality with
(\ref{eq:tec3}), we get
\begin{equation}\label{eq:tec4}
\|x^{k+1} -\hat{x}\|^2\leq \|x^{k} -\hat{x}\|^2 + 2t_{k}
\hat{\beta}|\langle J_{F}(x^k)^{T}\omega^k,v^k\rangle| .
\end{equation}
Since $K$ is a pointed, closed and convex cone,
$\mbox{int}(K^{\star})$ is a nonempty set; see \cite[Propositions
2.1.4, 2.1.7(i)]{Japonesa}. Therefore, $K^{\star}$ contains a basis
of $\RR^{m}$. Without loss of generality, we assume that the basis
\{$ \widetilde{\omega}^1,...,\widetilde{\omega}^m\} \subset
\mbox{conv}(G)$. Thus, for each $k$, there exist $\eta^{k}_{i} \in
\RR$, $i=1,\ldots,m$, such that $ \omega^{k}=
\sum_{i=1}^m\eta_{i}^{k}\widetilde{\omega}^i. $ By the compactness
of conv($G$),  there exists $L >0$, such that
 $|\eta_{i}^{k}|\leq L$ for all $i$ and $k$. Thus, inequality (\ref{eq:tec4}) becomes
$ \|x^{k+1} -\hat{x}\|^2\leq \|x^{k} -\hat{x}\|^2 +2t_{k}
\hat{\beta}L\sum_{i=1}^m |\langle\widetilde{\omega}^i,
J_{F}(x^k)v^k\rangle|.$

Defining $\delta_{k} := 2t_{k}
\hat{\beta}L\sum_{i=1}^m|\langle\widetilde{\omega}^i, J_{F}(x^k)v^k
\rangle|,$ it follows that $\delta_{k}> 0$. Since $(F(x^k))_{k\in
\NN}$ is $K$--bounded and using (\ref{somatorio}), we have
$\sum_{k=0}^{\infty}\delta_{k}<\infty$. Therefore, since $\hat{x}$
is an arbitrary element of $T$, the desired result follows from
Definition \ref{def:tec1}. \end{proof} Next theorem establishes a
sufficient condition for the convergence of the sequence
$(x^k)_{k\in \NN}$.

\begin{thm}\label{teo-1}
Assume that $F$ is a $K$--quasiconvex function. Then, $(x^k)_{k\in
\NN}$ converges to a stationary point.
\end{thm}
\begin{proof}
Since $F$ is $K$--quasiconvex, from Lemma \ref{Fejconv2} it follows
that  $ (x^k)_{k\in \NN}$ is quasi-Fej\'er convergent and, hence,
bounded; see Lemma \ref{fejer}(i). Therefore $(x^k)_{k\in \NN}$ has
at least one accumulation point, say $x^*$. From Proposition
\ref{iusmau}, $x^*$ is a stationary point. Moreover, since $C$ is
closed and the sequence is feasible, $x^*\in C$.

We proceed to prove $x^*\in T$. Since $F$ is continuous
$(F(x^k))_{k\in \NN}$ has $F(x^*)$ as accumulation point. By
Proposition~\ref{feasible}, $(F(x^k))_{k\in \NN}$ is a
$K$--decreasing sequence. Hence, the whole sequence $(F(x^k))_{k\in
\NN}$ converges to $F(x^*)$  and holds $ F(x^*)\preceq_K F(x^k)$
with $k\in \NN$, which implies that $x^*\in T$. Therefore, the
desired result follows from Lemma \ref{fejer}(ii) and
Proposition~\ref{iusmau}. \end{proof} This theorem extends the full
convergence obtained under $K$--convexity in \cite{Fukudainexact} to
the $K$--quasiconvex case. This class is actually  larger than
$K$--convex problems as next example shows.
\begin{exampleee}
Let $F: \RR\to\RR^2$, $F(t)=(4t^2,t^4-4t^2+2)$, and
$K=co(conv(\{(1,0),(1,1)\}))$. The function $F$  is not $K$--convex
because $ (F(0) +F(1))/2 - F(1/2)= (1,-9/16)\notin K, $ but, as
$\langle(1,0),F(t)\rangle = 4t^2$  and $\langle(1,1), F(t)\rangle =
t^4 + 2, $ are quasiconvex, by Theorem \ref{characterization}, $F$
is $K$--quasiconvex.
\end{exampleee}

\section{Final Remarks}\label{remarksss}

In this paper we considered the inexact projected gradient method
presented in \cite{Fukudainexact}. We explored strongly the
structure of problem \eqref{Prel_P}, mainly the quasi–convexity of
the function $F$, and obtained that the sequence generated by the
approach converges to a stationary point. So,  the method will be
successful for a class which is larger than the cases studied so
far.

 Future research is focused into two directions:  the practical
implementation of the method and its generalization to other cases.
In particular, we are looking for the convergence of a subgradient
method for solving the non-differentiable and $K$--quasiconvex
problem without to use scalarizations.

Recently were published in \cite{yunier-luis,yunier1,Bento2013-2}
the subgradient approaches for solving vectorial and feasibility
problems; see also \cite{orizon-xavier} using strongly
scalarizations technics. We also pretend to extend these methods to
the variable ordering case; see, for instance, \cite{yunier-gema}.

%%%%%%%%%%%%%%%%%%%%%%%%%%%%%%%%%%%%%%%%

%In this paper we prove that the sequence generated by the inexact
%projected gradient method converges in the case of quasiconvex
%vectorial optimization problems. We eliminate the scalarization
%approach, exploring strongly the structure of problem
%\eqref{Prel_P}, which is very important in the quasi--convex case.
%As future research, we will focus in the study of a subgradient
%method for nonsmooth $K$--quasiconvex functions without to use
%scalarizations. Recently were published in
%\cite{yunier1,Bento2013-2} the convergence analysis for nonsmooth
%$K$--convex functions and we also pretend to extend the method for
%variable ordering approach; see, for instance,
%\cite{yunier-gema,yunier-luis}.
%
%{\color{red} Agrandar mais esta parte...Glaydston pode mensionar seu
%recentes trabalhos em Variedades...e possiviliaddes nesse
%tema....aplicacoes na economia...etc...Glaydston eh o especialista
%nessa parte...citar aqui \cite{Bento2013-1, Bento2013-2}}

%%%%%%%%%%%%%%%%%%%%%%%%%%%%%%%%%%%%%%%%%%%%%%%%%%%%%%%%%%%%%%%%%%%%%%%%%%%%%%%%%%%%%%%%%%%

%%%%%%%%%%%%%%%%%%%%%%%%%%%%%%%%%%%%%%%%%%%%%%%%%%%%%%%%%%%%%%%%%%%%%%%%%%%%%%%%%%%%%%%%%%%%
\subsection*{Acknowledgment} The authors were partially supported by
CNPq, by projects PROCAD-nf - UFG/UnB/IMPA, CAPES-MES-CUBA 226/2012
and Universal-CNPq.

{\footnotesize

\bibliographystyle{plain}

}
%\bibliography{Yunierbib}
\end{document}